\documentclass[11pt,oneside,a4paper,abstract=on,DIV=12]{scrartcl}
\usepackage{overpic}
\usepackage{multicol}
\usepackage[bottom]{footmisc}

\usepackage{lmodern}
\usepackage[T1]{fontenc}
\usepackage[utf8]{inputenc}
\usepackage{amsmath,amssymb,amsfonts}
\usepackage{commath}
\usepackage{mathtools}
\usepackage{amsthm}

\usepackage{hyperref}
\hypersetup{
  pdfauthor={Johannes Siegele, Hans-Peter Schröcker, Martin Pfurner},
  pdftitle={Space Kinematics and Projective Differential Geometry Over the Ring of Dual Numbers},
  pdfkeywords={rational motion, motion polynomial, factorization, vertical Darboux motion, helical motion, osculating line, osculating conic, null cone motion, linkage},
  hidelinks,
}

\usepackage{todonotes}

\usepackage{url}
\newtheorem{proposition}{Proposition}[section]
\newtheorem{theorem}[proposition]{Theorem}
\newtheorem{lemma}[proposition]{Lemma}
\newtheorem{corollary}[proposition]{Corollary}
\theoremstyle{remark}
\newtheorem{example}[proposition]{Example}
\newtheorem{remark}[proposition]{Remark}

\newcommand{\C}{\mathbb{C}}
\newcommand{\D}{\mathbb{D}}
\newcommand{\F}{\mathbb{F}}
\newcommand{\DI}{\mathbb{D}^\times}
\newcommand{\RI}{\mathbb{R}^\times}
\renewcommand{\H}{\mathbb{H}}
\renewcommand{\DH}{\mathbb{DH}}
\newcommand{\DHI}{\mathbb{DH}^\times}
\renewcommand{\P}{\mathbb{P}}
\newcommand{\R}{\mathbb{R}}
\newcommand{\PD}[1][3]{\P^{#1}(\D)}
\newcommand{\SE}[1][3]{\mathrm{SE}(#1)}
\newcommand{\qi}{\mathbf{i}}
\newcommand{\qj}{\mathbf{j}}
\newcommand{\qk}{\mathbf{k}}
\newcommand{\eps}{\varepsilon}
\newcommand{\Cj}[1]{{#1}^\star}
\newcommand{\SQ}{\mathcal{S}}
\newcommand{\EG}{E}
\newcommand{\NC}{\mathcal{N}}

\title{Space Kinematics and Projective Differential Geometry Over the Ring of Dual Numbers\textsuperscript{$\ast$}}
\author{Johannes Siegele, Hans-Peter Schröcker, Martin Pfurner\\
  Department of Basic Sciences in Engineering Sciences\\
  University of Innsbruck, Austria}

\setkomafont{title}{\normalfont\bfseries}
\setkomafont{sectioning}{\normalfont\bfseries}

\begin{document}

\maketitle

\insert\footins{\footnotesize \textsuperscript{$\ast$}Extended version of the
  paper ``Space Kinematics and Projective Differential Geometry Over the Ring of
  Dual Numbers'', published in ICGG 20200 -- Proceedings of the 19th
  International Conference on Geometry and Graphics, Springer, 2021.}

\begin{abstract}
  We study an isomorphism between the group of rigid body displacements and the group of dual quaternions modulo the dual number multiplicative group from the viewpoint of differential geometry in a projective space over the dual numbers. Some seemingly weird phenomena in this space have lucid kinematic interpretations. An example is the existence of non-straight curves with a continuum of osculating tangents which correspond to motions in a cylinder group with osculating vertical Darboux motions. We also suggest geometrically meaningful ways to select osculating conics of a curve in this projective space and illustrate their corresponding motions. Furthermore, we investigate factorizability of these special motions and use the obtained results for the construction of overconstrained linkages.   \\[1mm]
  \\[1mm]
   \end{abstract}

\section{Introduction}
\label{sec:introduction}

The eight-dimensional real algebra $\DH$ of dual quaternions provides a
well-known model for the group $\SE$ of rigid body displacements. Dual
quaternions with non-zero real norm represent elements of $\SE$ and are uniquely
determined up to real scalar multiples. In the projectivization $\P(\DH) \cong
\P^7(\R)$ they correspond to points of the Study quadric $\SQ$ minus an
exceptional subspace $\EG$ of dimension three \cite{husty12}.

The Study quadric model provides a rich geometric and algebraic environment for
investigating questions of space kinematics. However, its ``curved'' nature
poses serious problems in numerous applications. One way of getting around this
is to consider dual quaternions modulo multiplication by \emph{dual numbers}
instead of just real numbers. The locus of the ensuing geometry is then not the
set $\SQ \setminus \EG \subset \P^7(\R)$ but the projective space $\P^3(\D)$ of
dimension three over the dual numbers (minus a low dimensional subset). It
provides a \emph{linear} model of space kinematics which is certainly a big
advantage. However, it also comes with some rather counter-intuitive properties:
The connecting straight line of two points is no longer unique and there exist
curves with an osculating tangent in any of their points.

What seems rather strange from a traditional geometric viewpoint becomes much
more natural in a kinematic interpretation where straight lines in $\PD$
correspond to vertical Darboux motions. Two poses may be interpolated by an
infinity of vertical Darboux motions \cite{schroecker18} and motions in cylinder
groups, for example helical motions, admit osculating Darboux motions at any
instance. We demonstrate and illustrate this in
Section~\ref{sec:osculating-lines}.

In Section~\ref{sec:osculating-conics} we present some results on osculating
conics/motions. Generically, there exists a four dimensional set of osculating
conics in every curve point. Among them we find the well-known Bennett motions
\cite{brunnthaler05,hamann11} but we also suggest another type of osculating
conic with geometric significance. It allows an interpretation as osculating
circle in an elliptic geometry. We further investigate factorizations of
polynomial parametric representations of these motions and give a geometric
criterion for factorizability. The obtained results are used to construct
overconstrained linkages whose couplers perform such motions.

\section{Preliminaries}
\label{sec:preliminaries}

A dual number is an element of the factor ring $\D \coloneqq \R[\eps]/\langle
\eps^2 \rangle$. It is uniquely represented by a linear polynomial $a + \eps b$
in the indeterminate $\eps$ with coefficients $a$, $b \in \R$, the \emph{primal}
and \emph{dual part,} respectively. Sum and product of two dual numbers as
implied by this definition are
\begin{equation*}
  (a + \eps b) + (c + \eps d) = a + c + \eps(b + d),\quad
  (a + \eps b)(c + \eps d) = ac + \eps(ad + bc).
\end{equation*}
Multiplication obeys the rule $\eps^2 = 0$. Provided $a \neq 0$, the
multiplicative inverse of $a + \eps b$ exists and is given by $(a+\eps b)^{-1} =
a^{-1} - \eps ba^{-2}$. We denote the set of invertible dual numbers by~$\DI$.

\subsection{Projective Geometry over Dual Numbers}
\label{sec:projective-geometry}
Similar to common projective geometry over the real or complex numbers, we
can study projective geometry over the dual numbers. We focus on the projective
space $\PD$ of dimension three over the dual numbers as this will be the
relevant case for rigid body kinematics. The elements of $\PD$ are
equivalence classes of elements of $\D^4 \setminus \{0\}$ where two vectors $x$
and $y$ are considered equivalent if there exists an invertible dual number $a +
\eps b$ such that $(a + \eps b)x = y$. We denote equivalence classes by square
brackets, i.\,e. as $[x]$ where $x \in \D^4$ or as $[x_0,x_1,x_2,x_3]$ where $x_0$,
$x_1$, $x_2$, $x_3 \in \D$.

In spite of its formal similarity with $\P^3(\R)$ or $\P^3(\C)$, the space $\PD$
exhibits some rather unusual properties. Let us consider the connecting line of two points $[a]$ and $[b]$. For its definition we already have two
choices. It can be considered as point set
\begin{equation}
  \label{eq:1}
  \{ [\alpha a + \beta b] \mid (\alpha,\beta) \in \F^2, (\alpha,\beta) \neq (0, 0) \}
\end{equation}
where $\F = \R$ or $\F = \D$, respectively. We will reserve the word
\emph{straight line} for the case $\F = \R$. There are two reasons for this
preference: Firstly, it seems to be the common notion in projective geometry
over rings. Secondly, a straight line in this sense has real dimension one while
it has real dimension two otherwise. With regard to kinematics, this means that
a straight line describes an, again more common, one-parametric motion.

A first, possibly surprising, geometric property of $\PD$ refers to the
connecting straight lines of two points. In contrast to geometry over the real
numbers, it is no longer unique.

\begin{proposition}
  Let $c$ and $d\in\D^4$ be two points such that $c$ or $d$ has at least one
  entry with non-zero primal part. Provided $[c]$ and $[d] \in \PD$ do not
  coincide, they have infinitely many connecting straight lines. If only one of
  the points $c$ and $d$ has an entry with non-zero primal part, the real
  dimension of the set of all connecting straight lines is equal to one,
  otherwise the dimension is two.
\end{proposition}

\begin{proof}
  We may parameterize any straight line connecting the given points by
  \eqref{eq:1} where $a = \gamma c$, $b = \delta d$ and $\gamma$, $\delta \in
  \DI$. This gives four real parameters, the coefficients of $\gamma$ and
  $\delta$. But multiplying $c$ and $d$ simultaneously with the same invertible
  dual number yields identical lines. Thus, only two essential real parameters
  remain. If all entries of $c$ or $d$ have primal part zero, the dual part of $\gamma$ or $\delta$ does not affect the product $\gamma c$ or $\delta d$, respectively. Thus we only have one essential parameter.

\end{proof}

\subsection{Space Kinematics}
\label{sec:space-kinematics}

A \emph{quaternion} is an element of the algebra $\H$ generated by basis
elements $1$, $\qi$, $\qj$, $\qk$ with generating relations $\qi^2 = \qj^2 =
\qk^2 = \qi\qj\qk = -1$ over the real numbers. A \emph{dual} quaternion $q$ is
an element of the algebra $\DH$ with the same basis elements and generating
relations but over the \emph{dual numbers} $\D$. Thus, we may write $q = q_0 + q_1\qi
+ q_2\qj + q_3\qk$ with $q_0$, $q_1$, $q_2$, $q_3 \in \D$ or, separating primal
and dual parts, $q = p + \eps d$ where $p = p_0 + p_1\qi + p_2\qj + p_3\qk$ and
$d = d_0 + d_1\qi + d_2\qj + d_3\qk$ are elements of~$\H$.

The conjugate dual quaternion is $\Cj{q} = q_0 - q_1\qi - q_2\qj - q_3\qk =
\Cj{p} + \eps\Cj{d}$, the dual quaternion norm is $q\Cj{q}$. In terms of
(coefficients of) $p$ and $d$ it may be written as $q\Cj{q} = p\Cj{p} +
\eps(p\Cj{d} + d\Cj{p}) = p_0^2 + p_1^2 + p_2^2 + p_3^2 + 2\eps(p_0d_0 + p_1d_1
+ p_2d_2 + p_3d_3) \in \D$. A dual quaternion $q$ is invertible if and only if
its norm is invertible as a dual number. Its inverse is then given by
$q^{-1}=(q\Cj{q})^{-1}\Cj{q}$. The unit norm condition for dual quaternions
reads as
\begin{equation*}
  p\Cj{p} = 1,\quad
  p\Cj{d} + d\Cj{p} = 0.
\end{equation*}
Because the norm is multiplicative, the unit dual quaternions form a
multiplicative group $\DH_0^\times$. We embed $\R^3$ into $\DH$ via
$(x_1,x_2,x_3) \hookrightarrow 1 + \eps(x_1\qi + x_2\qj + x_3\qk)$ and define
the action of $q = p + \eps d \in \DH_0^\times$ on points of $\R^3$ in the usual
way as
\begin{equation}
  \label{eq:2}
  1 + \eps(x_1\qi + x_2\qj + x_3\qk) \mapsto 1 + \eps(y_1\qi + y_2\qj + y_3\qk) = (p - \eps d) x (\Cj{p} + \eps\Cj{d})
\end{equation}
\cite{husty12}. This action provides us with a double cover of $\SE$, the group
of rigid body displacements, by the group $\DH_0^\times$.

A slight modification of \eqref{eq:2} extends the action to points
$[x_0,x_1,x_2,x_3]$ in the projective extension $\P^3(\R)$ of~$\R^3$:
\begin{equation}
  \label{eq:3}
  [x_0 + \eps(x_1\qi + x_2\qj + x_3\qk)] \mapsto [y_0 + \eps(y_1\qi + y_2\qj + y_3\qk)] = [(p - \eps d) x (\Cj{p} + \eps\Cj{d})].
\end{equation}
This gives an isomorphism between the group of dual quaternions of non-zero real
norm modulo $\RI$, the real multiplicative group, and $\SE$. The unit norm
condition of \eqref{eq:2} is replaced by the condition that the norm of $q$ be
real but non-zero:
\begin{equation}
  \label{eq:4}
  p\Cj{d} + d\Cj{p} = 0,\quad p\Cj{p} \neq 0.
\end{equation}
This means that $[q] = [p + \eps d]$ is a point on the so-called \emph{Study
  quadric $\SQ$} given by the quadratic form $p\Cj{d} + d\Cj{p}$, minus the
\emph{null cone $\NC$} given by the singular quadratic form $p\Cj{p}$. The only
real points of $\NC$ are those of its three-dimensional vertex space $\EG$,
given by $p = 0$. We call it the \emph{exceptional generator.}

A crucial observation for this article is that even the real norm requirement
\eqref{eq:4} can be abandoned: As long as $q\Cj{q}$ is invertible, \eqref{eq:3}
will describe a valid action on $\P^3(\R)$ and provide a \emph{homomorphism}
from the group $\DHI$ of invertible dual quaternions modulo $\RI$ to $\SE$ or
even an \emph{isomorphism} between $\DHI/\DI$ and~$\SE$.

\begin{proposition}
  The groups $\DHI/\DI$ and $\SE$ are isomorphic via the
  action~\eqref{eq:3}.
\end{proposition}

\begin{proof}
  It is easy to see that $\DHI$ is homomorphic to $\SE$ via \eqref{eq:3}. In
  order to see that $\DHI/\DI$ is isomorphic, we have to show that dual
  multiples yield the same action and that identical action implies a dual
  factor.

  Using the notation $q_\eps \coloneqq p - \eps d$ for the $\eps$-conjugate of
  $q = p + \eps d$ we can write the right-hand side of \eqref{eq:3} as $q_\eps x
  \Cj{q}$. Multiplying $q$ with a dual number $a$ yields $(aq)_\eps x
  \Cj{(aq)} = a_\eps q_\eps x a \Cj{q} = (a_\eps a) q_\eps x \Cj{q}$. Because
  $a_\eps a$ equals the squared primal part of $a$, this does not change the
  action on $\P^3(\R)$. Existence of a dual factor from identical action follows
  from equal dimension of $\DHI/\DI$ and $\SE$ and the fact that these groups
  have only one connected component.
\end{proof}

Since all elements of $\DHI/\DI$ are points of $\P^3(\D)$, it is natural to
study space kinematics via the projective geometry of $\P^3(\D)$. This point of
view is not new. It played a role in \cite{pfurner16} and \cite{purwar10}. From
an old paper by C.~Segre \cite{segre12} we even infer that probably already
E.~Study and his disciples were aware of these connections in the first decades
of the 20th century.

\subsection{Straight Lines and Vertical Darboux Motions}
\label{sec:straight-lines}

Via the action \eqref{eq:3}, a curve in $\PD$ corresponds to a one-parametric
rigid body motion. In particular, polynomial curves yield motions with
polynomial trajectories in homogeneous coordinates, that is, \emph{rational
  motions.} The simplest example of such motions comes from straight lines in
$\PD$ which correspond to \emph{vertical Darboux motions}
\cite{purwar10,scharler17}. A vertical Darboux motion is the composition of a
unit speed rotation about a fixed axis with a harmonic oscillation along the
axis such that one full rotation corresponds to one oscillation period. Its
trajectories are bounded rational curves of degree two (ellipses or
straight-line segments). Rotations and translations are considered as special
cases of vertical Darboux motions with zero or infinite oscillation amplitude,
respectively.

\begin{figure}[tb]
  \centering
  \includegraphics[page=48]{img/darboux-motion.pdf}
  \includegraphics[]{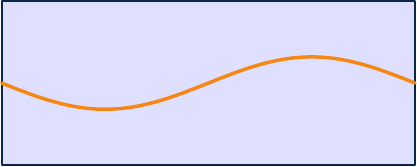}
  \caption{Vertical Darboux motion with some elliptic trajectories, right circular cylinder, and development.}
  \label{fig:darboux-motion}
\end{figure}

We illustrate a vertical Darboux motion in Figure~\ref{fig:darboux-motion}. This
figure also helps us explain a generally useful concept: Motions obtained as
composition of rotation around an axis and translation along the same axis have
trajectories on a right circular cylinder. Any curve $\gamma$ on such a cylinder
can be used to completely specify the motion by adding a Cartesian frame
consisting of cylinder normal, cylinder generator and horizontal cylinder
tangent. Instead of the curve on the cylinder, we may equally well consider its
image when developing the cylinder surface. In case of a vertical Darboux
motion, $\gamma$ is an ellipse. Its development is a sine curve which is scaled
in direction of the developed cylinder generators in order to adapt to the
oscillation's amplitude.

\section{Osculating Lines}
\label{sec:osculating-lines}

In this section we demonstrate that a helical motion and a vertical Darboux
motion can have second order contact at any parameter value. Since vertical
Darboux motions correspond to straight lines in $\PD$ this amounts to saying
that the curve corresponding to a helical motion has an osculating tangent at
any point. This is a remarkable difference to classical differential geometry
over the real numbers where this property characterizes straight lines. We also
prove that an infinity of osculating lines not only exists for a helical motion
but for any generic motion in a \emph{cylinder group,} i.\,e., a group generated
by rotations around and translations along a fixed axis.

A rotation around axis $\qk$ with rotation angle $\omega$ is given by the dual
quaternion $r = \cos\frac{\omega}{2} + \sin\frac{\omega}{2}\qk$, a translation with
oriented distance $\delta$ in direction of $\qk$ is given by $t = 1 -
\frac{1}{2}\eps \delta \qk$. Thus, a helical motion and a
Darboux motion $d$ with amplitude $c$ are obtained by substituting $p\omega$ and $c\sin\omega$,
respectively, for $\delta$ in the product~$rt$:
\begin{equation}
  \label{eq:5}
  \begin{aligned}
    h &= \cos\bigl( \tfrac{\omega}{2} \bigr) + \sin\bigl( \tfrac{\omega}{2} \bigr)\qk + \frac{p}{2}    \omega\eps(\sin\bigl( \tfrac{\omega}{2} \bigr)-\cos\bigl( \tfrac{\omega}{2} \bigr)\qk), \\
    d &= \cos\bigl( \tfrac{\omega}{2} \bigr) + \sin\bigl( \tfrac{\omega}{2} \bigr)\qk + \frac{c}{2}\sin\omega\eps(\sin\bigl( \tfrac{\omega}{2} \bigr)-\cos\bigl( \tfrac{\omega}{2} \bigr)\qk).
  \end{aligned}
\end{equation}
For $c=p$ the first two derivatives of $h$ and $d$ are equal,
\begin{equation*}
  \dod{h}{\omega}(0) = \dod{d}{\omega}(0) = \frac{1}{2}\qk - \frac{1}{2}p\eps\qk
  \quad\text{and}\quad
  \dod[2]{h}{\omega}(0) = \dod[2]{d}{\omega}(0) = -\frac{1}{4} + \frac{1}{2}p\eps,
\end{equation*}
while the third derivatives differ,
\begin{equation*}
  \dod[3]{h}{\omega}(0) = -\frac{1}{8}\qk + \frac{3}{8}p\eps\qk
  \neq -\frac{1}{8}\qk + \frac{7}{8}p\eps\qk = \dod[3]{d}{\omega}(0).
\end{equation*}
Thus, for $p = c$, the motions \eqref{eq:5} have second order contact at $\omega
= 0$. Since this parameter value has no particular meaning for a helical motion,
we may state that \emph{for any instance of a helical motion there exists a
  vertical Darboux motion with second order contact.}

Let us also verify that $d$ is actually a straight line in $\PD$ by multiplying
its parametric representation \eqref{eq:5} with a suitable dual number valued
function. Indeed, we have
\begin{equation*}
  \bigl(1 + p\eps\cos^2\bigl( \tfrac{\omega}{2} \bigr)\bigr)d
  = \cos\bigl( \tfrac{\omega}{2} \bigr)(1+p\eps) + \sin\bigl( \tfrac{\omega}{2} \bigr)\qk
\end{equation*}
which is a parametric representation of the straight line spanned by $1+p\eps$
and $\qk$. Summarizing, we can thus state

\begin{theorem}
  At any instance in time any helical motion, viewed as a curve in kinematic
  space $\PD$, has second order contact with a straight line. Yet, it is not a
  straight line itself.
\end{theorem}

This seemingly strange behavior allows a clear geometric interpretation that
also gives additional insight. Figure~\ref{fig:osculating-lines} displays
helical motion and osculating Darboux motion via the cylinder model we discussed
earlier. In the development, the helical motion corresponds to a straight line
while the Darboux motion is a sine curve with this line as inflection tangent.
Obviously, it is possible to determine uniquely a suitable sine function in
every point. It gives rise to the unique osculating vertical Darboux motion in a
point of the helical motion.

\begin{figure}[tb]
  \centering
  \includegraphics{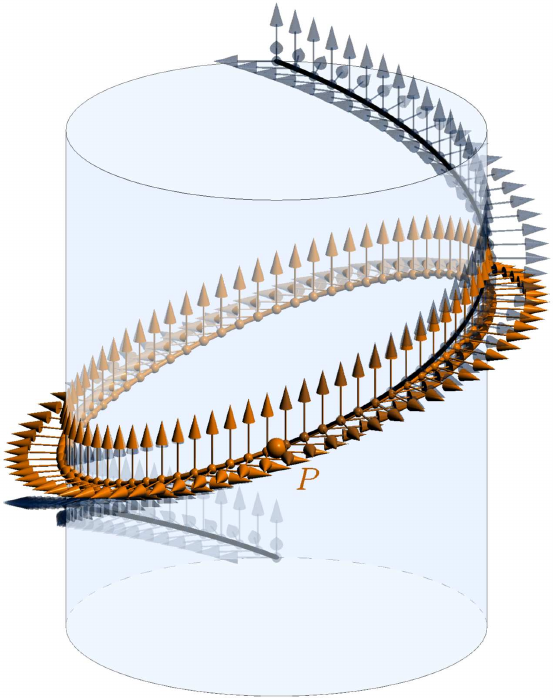}
  \includegraphics{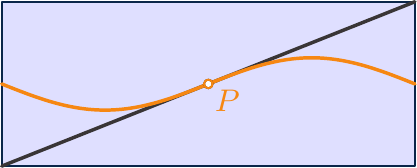}
  \caption{Geometric interpretation of osculating lines.}
  \label{fig:osculating-lines}
\end{figure}

Helical motions are not the only curves in $\PD$ susceptible to second order
approximation by straight lines in every point. An arbitrary motion in the
cylinder group $C$ corresponds to a curve in the development. There, an
osculating sine function can be drawn in any sufficiently smooth point and gives
rise to an osculating vertical Darboux motion. The possibility to do so is a
direct consequence of the following lemma: Among all candidate sine functions
there exist one with prescribed slope and curvature.

\begin{lemma}
  \label{lem:aphi}
  Given two real numbers $k$, $\varkappa \in \R$, there exists $a \in \R$ such
  that some point on the graph of the function $\varphi \mapsto a\sin\varphi$
  has slope $k$ and curvature~$\varkappa$.
\end{lemma}

\begin{proof}
  The subgraphs corresponding to parameter intervals $[i\frac{\pi}{2},
  (i+1)\frac{\pi}{2}]$ for $i \in \{0,1,2,3\}$ are congruent and, up to
  respective signs, have points of equal slope and curvature. Thus, we may
  restrict to the case $i = 0$, $k \ge 0$, $\varkappa \le 0$ and search for $a >
  0$. Slope and curvature are given by
  \begin{equation*}
    k = a\cos\varphi \quad\text{and}\quad \varkappa = -\frac{a\sin\varphi}{(1+a^2\cos^2\varphi)^{3/2}}
  \end{equation*}
These identities can be written in the following way
\begin{align*}
	k = a\cos\varphi \quad\text{and}\quad -\varkappa(1+k^2)^{3/2} = a\sin\varphi.
\end{align*}
This allows us to view $a$ and $\varphi$ as the polar coordinates of the point
$(k,-\varkappa(1+k^2)^{3/2})$ which lies, due to our assumptions, in the first
quadrant. As its polar coordinates are determined (and even unique if $k$ and
$\varkappa$ are not both zero), we can find $a\ge 0$ and
$\varphi\in[0,\frac{\pi}{2}]$.

\end{proof}

\begin{corollary}
  Any sufficiently smooth motion in a cylinder group has an osculating Darboux
  motion in any of its points (at any instance).
\end{corollary}

\section{Conic Sections}
\label{sec:osculating-conics}

We now turn our attention to conic sections in $\PD$. We study them as rational
curves of degree two. A parametric representation is simply a polynomial $c$ of
degree two in one indeterminate $t$ that serves as a real parameter. We assume
that $c$ has no scalar polynomial factor of positive degree and also that its
coefficients are linearly independent, as otherwise it
would parametrize a straight line or a point. A conic parameterizes a rational motion with
trajectories of degree at most four.

In line with the general philosophy of this article we should consider a
polynomial $c$ up to multiplication with a dual number valued function. However,
here it is sufficient to consider only dual number multiples, that is,
we consider only polynomial representations of minimal degree. In projective differential
geometry over the real numbers, a generic smooth space curve admits a two
parametric set of osculating conics in a generic point. In projective geometry
over the dual numbers, a further degree of freedom is added:

For the case of interpolating conics for three finitely separated points
$[c_0]$, $[c_1]$, $[c_2] \in \PD$, this is easy to see. An interpolating conic
may be parameterized as $[c(t)]$ where
\begin{equation*}
  c(t) = c_0 + (c_1 - c_0 - c_2)t + c_2t^2.
\end{equation*}
The points $[c_0]$, $[c_1]$, $[c_2]$ correspond to parameter values $t = 0$, $t
= 1$, and $t = \infty$, respectively. Obviously, different dual number multiples
of $c_0$, $c_1$ and $c_2$ yield different conics, unless the dual factor is the
same for all three points. We may use this freedom to have the dual factor $1$
for $c_1$ whence a general parametric representation for interpolating conics
can be written as
\begin{equation}
  \label{eq:6}
  c(t) = \gamma_0c_0 + (c_1 - \gamma_0c_0 - \gamma_2c_2)t + \gamma_2c_2t^2
\end{equation}
where $\gamma_0$ and $\gamma_2$ are invertible dual numbers. This gives the
claimed four degrees of freedom.

In view of Section~\ref{sec:osculating-lines} it is natural to ask for space
curves that admit a conic with even higher order contact in every point. We will
not pursue this question any further at this point. Instead, we present two
examples of osculating conics in this set with a special meaning for space
kinematics. In Section~\ref{sec:factorization} we discuss their factorization in
the sense of \cite{hegedus13} and use it for the construction of linkages.

\subsection{Bennett Motions}

The \emph{Bennett motion} is a well-known example of a quartic space motion
whose kinematic image in the ``classical'' sense is a conic section on the Study
quadric $\SQ$ and which is determined by three general finitely separated or
infinitesimally neighboring points in the Study quadric. In fact, we may simply
define it as a regular conic in the Study quadric that does not intersect the
exceptional generator $\EG$ \cite{brunnthaler05,hamann11}. In our context, we
can re-derive the motion from the following observation:

\begin{lemma}
  \label{lem:2}
  Given an invertible dual quaternion $p$ there exists an invertible dual number
  $a$ such that $ap$ has real norm. The dual number $a$ is determined up to a
  real multiple.
\end{lemma}

\begin{proof}
  Write $p = p' + \eps p''$ and $a = a' + \eps a''$ with quaternions $p'$, $p''$
  and real numbers $a'$, $a''$. The dual part of the norm of $ap$ then reads as
  $a'^2(p'\Cj{p''}+p''\Cj{p'}) + 2a'a''p'\Cj{p'}$. Both, $p'\Cj{p''} +
  p''\Cj{p'}$ and $p'\Cj{p'}$ are real numbers and the latter is different from
  zero (because $p$ is invertible). We may divide by $a'$ (because we want to find
  an invertible dual number $a$) so that ultimately $a = a' + \eps a''$ is
  determined, up to a real multiple, by one non-vanishing homegeneous linear
  equation. A (non-trivial) solution with $a' = 0$ is not possible because $p$ is invertible
  whence $p'\Cj{p'} \neq 0$.
\end{proof}

Returning to \eqref{eq:6}, we may assume $[c_0]$, $[c_1]$, $[c_2] \in \SQ$ as
otherwise we can multiply with suitable dual numbers by Lemma~\ref{lem:2}. Now
we are still free to multiply $c_0$, $c_1$, and $c_2$ with real numbers and it
is well-known (c.~f. for example \cite{brunnthaler05}) that this freedom is
enough to ensure that $[c(t)]$ lies on the Study quadric~$\SQ$.

Bennett motions are rational motions with entirely circular trajectories of
degree four. They appear as coupler motions of \emph{Bennett linkages}, that is,
spatial four-bar linkages with exceptional mobility \cite{brunnthaler05}. An
example for a Bennett motion can be found later in Figure~\ref{fig:osculating-motions}.

\subsection{Motions Based on Osculating Circles of Elliptic Geometry}

An important object for the kinematic geometry in $\PD$ is the null cone $\NC$.
It consists of points represented by non-invertible dual quaternions, a
property that is preserved under coordinate changes and thus makes $\NC$ a
\emph{geometric invariant}. With this in mind, it is natural to look for
osculating conics in special position with respect to $\NC$. For a general
parametric representation of the shape \eqref{eq:6} it is possible to determine
the dual factors $\gamma_0$, $\gamma_2 \in \D$ in such a way that the conic
parameterized by $[c(t)]$ is \emph{tangent to $\NC$ in two points.} In fact, if
we only consider \emph{real} factors, this amounts to determining a circle
through three points in the real elliptic plane with absolute conic $\NC \cap
\varphi$ where $\varphi$ is the conic's plane. For three finitely separated
points this problem has four solutions as can be seen in the spherical model of
elliptic geometry (Figure~\ref{fig:osculating-circles}). But this property does
not translate to three infinitesimally neighboring points as in the limit three
of the four circles converge to the curve tangent so that the osculating circle
is unique. This is also visualized in Figure~\ref{fig:osculating-circles}.

\begin{figure}[tb]
  \centering
  \includegraphics[page=1]{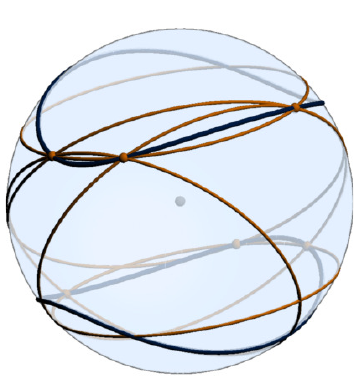}
  \includegraphics[page=20]{img/osculating-circles}
  \includegraphics[page=36]{img/osculating-circles}
  \caption{Circles through three points (left and middle) and osculating circle
    in elliptic geometry (right).}
  \label{fig:osculating-circles}
\end{figure}

In lack of a better name, we refer to the motions in question as \emph{quadratic
  null cone motions.} The four-dimensional set of osculating conics contains a
two-dimensional set of these motions. Their generic trajectories are rational of
degree four, not circular in general but tangent to the plane at infinity in two points. 
Figure~\ref{fig:osculating-motions} displays a null cone motion and a Bennett
motion that osculate at one pose which is drawn a little larger.

\begin{figure}[bt]
  \centering
  \begin{overpic}[trim=0 50 0 0, clip]{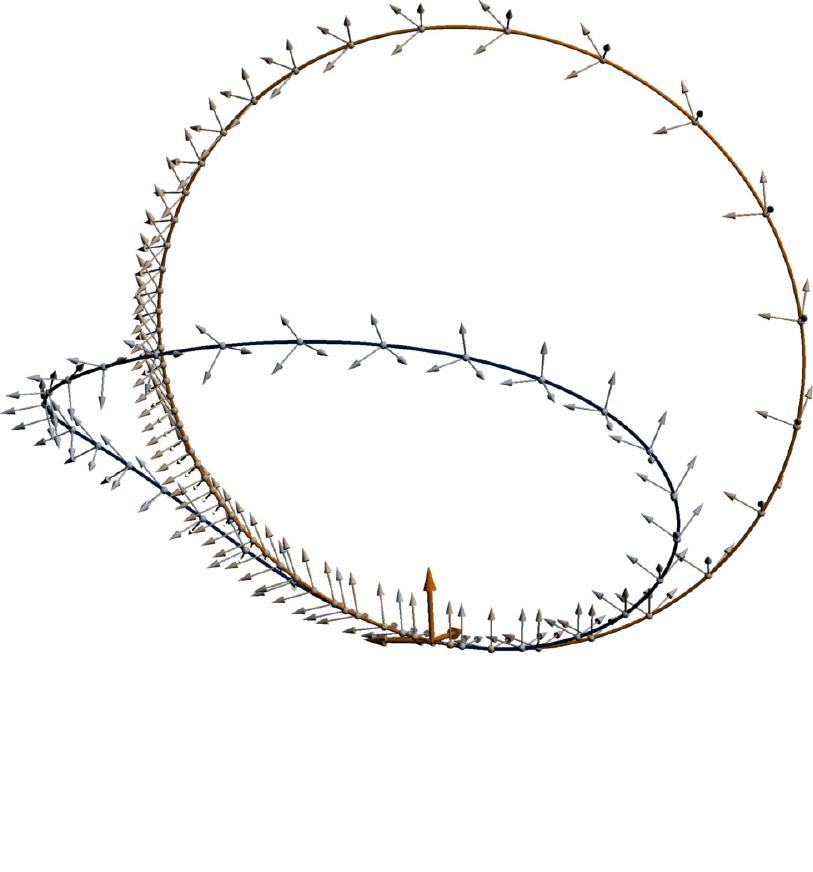}
    \put(40,72){\small Bennett motion}
    \put(28,49){\small osculating null cone motion}
  \end{overpic}
  \caption{A Bennet motion and an osculating null cone motion.}
  \label{fig:osculating-motions}
\end{figure}

\subsection{Factorization of Bennett Motions and Null Cone Motions}
\label{sec:factorization}

It is well known that a general conic section on the Study quadric (a Bennett
motion) occurs as the coupler motion of a closed spatial four-bar linkage
(Bennett linkage) \cite{brunnthaler05,hamann11}. One way to find the
corresponding linkage to a given Bennett motion is to compute different
factorizations of a parametric representation into linear polynomials over the
ring of dual quaternions. Each factor parametrizes a straight line on the Study
quadric and therefore, generically, a rotation \cite{hegedus13}. It can be
realized mechanically via a revolute joint. Combining the joints of two
factorizations yields the spatial four-bar linkage. For special conic sections
on the Study quadric, factorization and construction of linkages can be more
involved. It is possible that revolute joints have to be replaced by prismatic
(translation) joints \cite{hamann11,hegedus13}. It may also be the case that
only one or infinitely many factorizations exist.

The case of only one factorization is briefly mentioned in \cite{siegele20}.
This happens precisely if the Bennett motion is also a null cone motion. Since
essentially everything is known about the factorization of Bennett motions, we
now have a closer look at factorizability of null cone motions. We also discuss
how to construct a linkage to generate this motion, provided factorizations
exist. In $\P ^3(\D)$, straight lines no longer correspond to rotations but to
vertical Darboux motions with rotations and translations as special cases. Since
a vertical Darboux motion is a composition of a rotation and a harmonic
oscillation along the same axis, it can be realized by a cylindrical joint which
allows an independent rotation around and translation along a fixed axis. When
constructing linkages one will have to ensure that the cylindrical joints do not
introduce unwanted degrees of freedom.

Let $c(t)= c_0+(c_1-c_0-c_2)t+c_2t^2=p(t)+\eps d(t)$ be a conic section which is
tangent to $\NC$. We will assume that $c_2=1$ which can be achieved by a change
of coordinates. Since $c$ is tangent to the null cone in two points, we know
that the norm polynomial's primal part $p\Cj{p}$ is a real polynomial with two
roots of multiplicity two.

If the two roots coincide, i.e. $p\Cj{p}$ has one root with multiplicity four,
we deduce that this root is actually real. But then there exists a real
parameter value $t_0$ at which the norm of $p$ is zero whence $p(t_0) = 0$. This
implies that $p$ is the square of a linear polynomial $s\in\R[t]$. The
underlying motion is the translation along a quadratic curve with one point at
infinity, that is, a parabola or a quadratically parametrized half-line. A
factorization necessarily is of the form $c=(s+\eps d_1)(s+\eps d_2)$. But in
this case $s$ is a factor of $c$ and therefore $c$ does not parametrize a conic
but a line and it cannot be a quadratic translation. Therefore we assume for the
remainder of this section that the two roots of $p\Cj{p}$ are distinct.

\begin{theorem}
  \label{th:ncm-factorization}
  If $c = p + \eps d$ parametrizes a quadratic null-cone motion, a factorization
  exists precisely if $p\Cj{p}$ has two distinct roots $t_0$, $t_1 \in \C$ of
  multiplicity two and the points $[c(t_0)]$, $[c(t_1)]$ lie on the Study
  quadric.
\end{theorem}

\begin{proof}
  We distinguish between three cases:
  \begin{itemize}
  \item[(a)] $p$ does not have a real polynomial factor,
  \item[(b)] $p$ is an irreducible quadratic real polynomial, or
  \item[(c)] $p$ is the product of two distinct linear real polynomials.
  \end{itemize}
  The seemingly missing case of just one linear real factor is not compatible
  with the assumption of $c$ parametrizing a null cone motion. It can only happen in
  Case (a) that the points $[c(t_0)]$, $[c(t_1)]$ don't lie on the Study quadric
  since both (b) and (c) ensure that the dual part of $c\Cj{c}$ has a real
  polynomial factor with roots $t_0$, $t_1$. Thus it suffices to show the
  existence of factorizations in Case (b) and Case (c).

  Let us consider Case (a) at first. Here, the two roots $t_0$, $t_1$ are
  complex conjugates and we have $p\Cj{p}=s^2$ for the irreducible quadratic
  polynomial $s = (t - t_0)(t - t_1)\in\R[t]$ (recall that $c$ is assumed to be
  monic). An obviously necessary condition for $c$ to admit a factorization into
  linear polynomials is that $c\Cj{c}=s^2+\eps(p\Cj{d}+d\Cj{p})$ is a product of
  two quadratic polynomials with dual coefficients, that is $c\Cj{c} = (s+\eps
  \lambda_1)(s+\eps \lambda_2)$ for some $\lambda_1$, $\lambda_2\in\R[t]$. This
  implies $p\Cj{d}+d\Cj{p}=s(\lambda_1+\lambda_2)$ so that $s$ is a factor of
  $c\Cj{c}$ and the two points $[c(t_0)]$, $[c(t_1)]$ of tangency of the conic
  $c$ and the nullcone also lie on the Study quadric. For the converse
  statement, we appeal to \cite[Lemma~3]{hegedus13} which ensures existence of a
  factorization.

  Next, let us consider Case (b) where $p$ is an irreducible quadratic real
  polynomial, that is $c=p+\varepsilon d$ with $d \in \H[t]$. A suitable
  parameter transformation allows us to assume $p = t^2 + 1$. In this case, $c$
  describes a translation along a bounded quadratic curve (an ellipse or line
  segment). After a change of coordinates we can assume that all trajectories
  lie in planes orthogonal to the third coordinate axis and their major axes are
  parallel to the first coordinate axis. Thus, the parametrization is of the
  form $c=t^2+1+\varepsilon (\gamma_1t+ \gamma_0+b\qj t+a\qi)$, where $a\ge b
  \ge 0$, $a \neq 0$ are the lengths of their semi-axes and $\gamma_1$,
  $\gamma_0$ are arbitrary real numbers. As in \cite{Li2015}, we solve the
  equation $c=F_1 F_2$ for arbitrary linear factors $F_1$, $F_2\in\DH[t]$. Since
  the primal part of $c$ is a real polynomial, we get that the primal parts of
  $F_1$ and $F_2$ are conjugates of each other. Let us write
\begin{align*}
  F_1 &= t+p_1\qi + p_2\qj + p_3 \qk +\varepsilon(u_0 + u_1\qi + u_2\qj + u_3\qk),\\
  F_2 &= t-p_1\qi - p_2\qj - p_3 \qk +\varepsilon(v_0 + v_1\qi + v_2\qj + v_3\qk).
\end{align*}
Comparing coefficients in $c=F_1F_2$, we obtain a system of nine algebraic equations:
\begin{gather*}
  u_0+v_0 = \gamma_1,\quad  u_1+v_1 = 0,\quad  u_2+v_2 = b,\quad  u_3+v_3 = 0,\quad  p_1^2+p_2^2+p_3^2 =1,\\
  p_1(u_1-v_1) + p_2(u_2-v_2) + p_3(u_3-v_3) = \gamma_0,\quad
  p_1(v_0-u_0) + p_2(u_3+v_3) - p_3(u_2+v_2) = a,\\
  -p_1(u_3+v_3) + p_2(v_0-u_0) + p_3(u_1+v_1) = 0,\quad
  p_1(u_2+v_2) -p_2(u_1+v_1) + p_3(v_0-u_0) = 0.
\end{gather*}
For $a>b$, this system has two two parametric families of solutions:
\begin{equation}\label{sol:ell_trans}
\begin{gathered}
  p_1 = \pm\frac{\sqrt{a^2-b^2}}{a},\quad p_2 = 0,\quad p_3 = -\frac{b}{a},\quad
  u_2 = -v_2 + b,\quad u_3 = -v_3\\
  u_0 = \frac{\gamma_1}{2}\mp\frac{\sqrt{a^2-b^2}}{2},\quad u_1=-v_1= \pm\frac{a\gamma_0-2bv_3}{2\sqrt{a^2-b^2}},\quad
  v_0 =  \frac{\gamma_1}{2}\pm\frac{\sqrt{a^2-b^2}}{2}.
\end{gathered}
\end{equation}
If $a=b > 0$, which is the case if and only if $c$ is a circular translation, there
is only one family of solutions, namely
\begin{gather*}
  p_1=p_2=0,\quad p_3=-1,\quad u_0=v_0=\frac{\gamma_1}{2},\quad
  u_1=-v_1,\quad u_2=-v_2+b,\quad u_3=-v_3=-\frac{\gamma_0}{2}.
\end{gather*}
In case of $\gamma_0 = \gamma_1 = 0$, these solutions correspond to the ones
found in \cite{Li2015} for quadratic motion polynomials.

In the final Case (c) the polynomial $p$ is the product of two distinct linear
real polynomials $s_1$, $s_2\in\R[t]$. Let us find the reduced polynomial
$\widetilde{c}$ which parametrizes the same motion and fulfills the Study
condition. Following the proof of Lemma~\ref{lem:2} we can find the polynomial $
a= s_1s_2(2s_1s_2-\eps (d+\Cj{d})) \in \D[t]$ with dual number coefficients such
that $ac$ has a real norm polynomial. This product has the real polynomial factor
$(s_1s_2)^2$, so after reducing $ac$ and dividing off the leading coefficient we end
up with $\widetilde{c} \coloneqq (1-\eps(d+\Cj{d})/(2s_1s_2))c$. The norm
polynomial of $\widetilde{c}$ equals $s_1^2s_2^2$. Thus we can apply Lemma~3 of
\cite{hegedus13} to infer existence of a factorization $\widetilde{c} = F_1F_2$
such that $F_1\Cj{F}_1 = s_1^2$ and $F_2\Cj{F}_2 = s_2^2$. As $\widetilde{c}$ is
the product of $c$ with a rational dual function, we can multiply
$\widetilde{c}$ with the inverse of this function to obtain $c$, i.\,e.
$c=(1+\eps(d+\Cj{d})/(2s_1s_2))F_1F_2$. Let us use partial fraction
decomposition for the dual part of the first factor such that
$(1+\eps(d+\Cj{d})/(2s_1s_2))=(1+\eps\lambda_1/s_1)(1+\eps \lambda_2/s_2)$ for
some $\lambda_1$, $\lambda_2\in\R$. This now yields a factorization of $c$ as we
obtain $c=(1+\eps\lambda_1/s_1)F_1(1+\eps \lambda_2/s_2)F_2$.
\end{proof}

As we have seen in the proof, polynomials which fulfill Case (b) parametrize
bounded quadratic translations. In Case (c) they parametrize quadratic
translations along a curve which intersects the plane at infinity for two real
parameters. Thus, these polynomials parametrize translations along a hyperbola.

\begin{remark}
  The proof of Theorem~\ref{th:ncm-factorization} shows that all bounded
  quadratic translations admit a factorization. But only circular translations
  which fulfill the Study condition can be decomposed into two rotations, that
  is, they have factors which satisfy the Study condition. This confirms results
  of~\cite{Li2015}.
\end{remark}

The proof of Theorem~\ref{th:ncm-factorization} together with Lemma~\ref{lem:2}
shows that in Case~(b) and Case~(c) there exist infinitely many quadratic
polynomials which parametrize the same motion. In Case (b) choosing different
representations allows us to decompose the given motion into linear polynomials
in infinitely many different ways. In Case (c) however, the linear factors
obtained from different representations only differ by dual rational factors.
While the factorizations are different in algebraic sense, the underlying
kinematic decompositions are all the same. The obtained factors parametrize
translations and therefore can be realized by prismatic joints. As all quadratic
translations have trajectories parallel to a plane, they can always be realized
by coupling two prismatic joints with axes parallel to this plane but non-parallel to each
other.
Therefore, factorization does not allow for the construction of an
overconstrained linkage performing motions of type (c).

For Case (a) and (b) however we have seen in the proof of
Theorem~\ref{th:ncm-factorization} that we can find (at least) two different
factorizations. For the construction of overconstrained linkages we need
factorizations where one linear factor has a real norm polynomial, hence
parametrizes a rotation. The other factors are linear polynomials with non-real
norm polynomial and therefore parametrize vertical Darboux motions which can be
realized by cylindrical joints. Using these types of factorization we can
construct four-bar linkages consisting of two revolute and two cylindrical
joints which are able to perform the given motion. The degree of freedom
according to the formula of Chebyshev-Gr\"ubler-Kutzbach
\cite[Chapter~5]{angeles89} equals $0$, thus the linkages are overconstrained.

As a linkage corresponding to a motion of Case~(a) is described in \cite{siegele20}, we
focus on the construction of a linkage corresponding to Case~(b). Let $c=t^2+1+\varepsilon (\gamma_1t+
\gamma_0+b\qj t+a\qi)$ be an elliptic translation with $a>b\ge 0$. Note that
varying $\gamma_0$ and $\gamma_1$ yield different representations of the same
underlying motion. The linear factors obtained in the proof of
Theorem~\ref{th:ncm-factorization} have the norm polynomials
$t^2+1+\eps(t(\gamma_1\pm\sqrt{a^2-b^2})+\gamma_0)$. Thus, we can choose
$\gamma_0=0$ and $\gamma_1=\mp\sqrt{a^2-b^2}$ such that either the first, or the
second factor has real norm.

All factors of one family of factorizations parametrize motions around parallel
axes. Choosing two different factorizations from the same family therefore would
yield a four-bar linkage with two revolute and two cylindrical joints with
parallel axes. Such linkages however have in general two degrees of freedom.
Therefore we need to choose one factorization from each family. The two
different choices of $\gamma_1$ determine, which joints are the revolute joints.

\begin{example}
  Let us consider a translation along an ellipse with semi-major axis of length
  $a=2$ and semi-minor axis of length $b=1$. For $\gamma_1=\sqrt{a^2-b^2}$ we
  have $c=t^2+1+\varepsilon (\sqrt{3}t+\qj t+2\qi)$. Choosing $v_2=0$,
  $v_3=1/\sqrt{3}$ in Equation~\eqref{sol:ell_trans} we obtain the factorizations $c=F_1F_2=G_1G_2$ with
  \begin{align*}
    &F_1 = t + \frac{\sqrt{3}\qi-\qk}{2} - \varepsilon\frac{\qi-3\qj+\sqrt{3}\qk}{3},\quad
    &&F_2 = t - \frac{\sqrt{3}\qi-\qk}{2} + \varepsilon\frac{3\sqrt{3}+\qi+\sqrt{3}\qk}{3},\\
    &G_1 = t - \frac{\sqrt{3}\qi+\qk}{2} + \varepsilon\frac{3\sqrt{3}+\qi+3\qj-\sqrt{3}\qk}{3},\quad
    &&G_2 = t + \frac{\sqrt{3}\qi+\qk}{2} + \varepsilon\frac{-\qi+\sqrt{3}\qk}{3}.
  \end{align*}
  The axes of the motions parametrized by $F_1$ and $F_2$ as well as $G_1$ and
  $G_2$, respectively, are parallel. The angle between the non-parallel axes is
  $\pi/3$. The distance between the parallel axes is $1$ while the distance
  between the other axes is $4/3$. The obtained linkage is depicted in
  Figure~\ref{fig:RCRC-elliptic}. It can be shown that it has two operation
  modes, both are elliptic translations.
  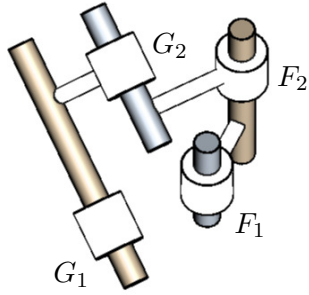
\begin{figure}[tb]
    \centering
    \begin{overpic}[trim=1.5cm 3cm 2.4cm 3.3cm, clip, scale=2]{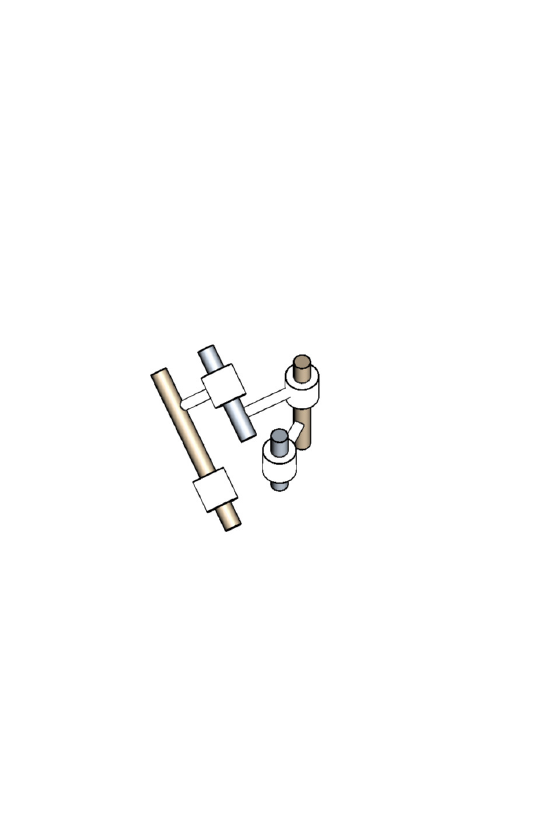}
\put(70,20){$F_1$}
\put(83,65){$F_2$}
\put(15,5){$G_1$}
\put(45,75){$G_2$}
\end{overpic}
    \caption{Spatial four-bar linkage with revolute joints $F_1$, $G_2$ and cylindrical joints $F_2$, $G_1$ whose coupler performs an elliptic translation.}
    \label{fig:RCRC-elliptic}
  \end{figure}
\end{example}

\section{Conclusion}
\label{sec:conclusion}

We have related space kinematics to the geometry of the projective space $\PD$
over the ring of dual numbers. This interpretation seems well suited for
kinematic visualization of certain differential geometric aspects and it also
provides the proper mathematical framework for the systematic study of
osculating motions. We presented results for ordinary and osculating tangents
and some ideas about osculating conics and their factorizability. Factorization
without the Study condition opens additional possibilities for the construction
of linkages with cylindrical joints.

\section*{Acknowledgment}

We thank the anonymous reviewers for their valuable suggestions that helped us
to improve the paper. This is in particular true for the proof of
Lemma~\ref{lem:aphi}. Johannes Siegele was supported by the Austrian Science
Fund (FWF): P~30673 (Extended Kinematic Mappings and Application to Motion
Design).

\bibliographystyle{plain}

\begin{thebibliography}{10}

\bibitem{angeles89}
Jorge Angeles.
\newblock {\em Rational Kinematics}.
\newblock Springer, 1989.

\bibitem{brunnthaler05}
Katrin Brunnthaler, Hans-Peter Schröcker, and Manfred Husty.
\newblock A new method for the synthesis of {Bennett} mechanisms.
\newblock In {\em Proceedings of CK 2005, International Workshop on
  Computational Kinematics}, Cassino, 2005.

\bibitem{hamann11}
M.~Hamann.
\newblock Line-symmetric motions with respect to reguli.
\newblock {\em Mech. Mach. Theory}, 46(7):960--974, 2011.

\bibitem{hegedus13}
G.~Hegedüs, J.~Schicho, and H.-P. Schröcker.
\newblock Factorization of rational curves in the {Study} quadric and revolute
  linkages.
\newblock {\em Mech. Mach. Theory}, 69(1):142--152, 2013.

\bibitem{husty12}
Manfred Husty and Hans-Peter Schröcker.
\newblock Kinematics and algebraic geometry.
\newblock In J.~Michael McCarthy, editor, {\em 21st Century Kinematics. The
  2012 NSF Workshop}, pages 85--123. Springer, London, 2012.

\bibitem{Li2015}
Z.~Li, , T.~Rad, J.~Schicho, and H.-P. Schröcker.
\newblock Factorization of rational motions: A survey with examples and
  applications.
\newblock In {S.-H. Chang et al.}, editor, {\em Proc. IFToMM 14}, pages
  833--840, 2015.

\bibitem{pfurner16}
Martin Pfurner, Hans-Peter Schröcker, and Manfred Husty.
\newblock Path planning in kinematic image space without the {Study} condition.
\newblock In Jadran Lenarčič and Jean-Pierre Merlet, editors, {\em
  Proceedings of Advances in Robot Kinematics}, 2016.

\bibitem{purwar10}
Anurag Purwar and Jeff Ge.
\newblock Kinematic convexity of rigid body displacements.
\newblock In {\em Proceedings of the ASME 2010 International Design Engineering
  Technical Conference \& Computers and Information in Engineering Conference
  IDETC/CIE}, Montreal, 2010.

\bibitem{scharler17}
Daniel Scharler.
\newblock Characterization of lines in extended kinematic image space.
\newblock Master thesis, University of Innsbruck, 2017.

\bibitem{schroecker18}
Hans-Peter Schröcker.
\newblock From {A} to {B: New} methods to interpolate two poses.
\newblock {\em J. Geom. Graphics}, 22(1):87--98, 2018.

\bibitem{segre12}
Corrado Segre.
\newblock Le geometrie proiettive nei campi di numeri duali.
\newblock In {\em Corrado Segre, Opere}, volume~II, pages 396--431. Unione
  Matematica Italiana, Roma, edizione cremonese edition, 1912.

\bibitem{siegele20}
Johannes Siegele, Martin Pfurner, and Hans-Peter Schröcker.
\newblock Factorization of dual quaternion polynomials without {Study's}
  condition.
\newblock Accepted for publication in Adv. Appl. Clifford Algebras, 2021.

\end{thebibliography}

\end{document}